\documentclass[11pt,fleqn]{article}
\usepackage[paper=a4paper]
  {geometry}

\usepackage[small]{titlesec}
\usepackage{paragraphs}

\usepackage{hyperref}

\usepackage{amsthm,thmtools}
\usepackage[utf8]{inputenc}
\usepackage[english]{babel}
\usepackage{enumerate}
\usepackage[osf,noBBpl]{mathpazo}
\usepackage[alphabetic,initials]{amsrefs}
\usepackage{amsfonts,amssymb,amsmath}
\usepackage{mathtools}
\usepackage{graphicx}
\usepackage[poly,arrow,curve,matrix]{xy}
\usepackage{wrapfig}
\usepackage{xcolor}
\usepackage{helvet}
\usepackage{stmaryrd}
\usepackage[normalem]{ulem}

\declaretheoremstyle[headformat=swapnumber, spaceabove=\paraskip,
bodyfont=\itshape]{mystyle}
\declaretheorem[name=Lemma, sibling=para, style=mystyle]{Lemma}
\declaretheorem[name=Proposition, sibling=para, style=mystyle]{Proposition}
\declaretheorem[name=Theorem, sibling=para, style=mystyle]{Theorem}

\declaretheorem[name=Definition, sibling=para, style=mystyle]{Definition}

\declaretheoremstyle[numbered=no, spaceabove=\paraskip,
bodyfont=\itshape]{mystyle-empty}
\declaretheorem[name=Lemma, style=mystyle-empty]{Lemma*}
\declaretheorem[name=Proposition, style=mystyle-empty]{Proposition*}
\declaretheorem[name=Theorem, style=mystyle-empty]{Theorem*}
\declaretheorem[name=Corollary, style=mystyle-empty]{Corollary*}
\declaretheorem[name=Definition, style=mystyle-empty]{Definition*}
\declaretheorem[name=Examples, style=mystyle-empty]{Example*}
\declaretheorem[name=Remark, style=mystyle-empty]{Remark*}

\declaretheoremstyle[
        headformat={{\bfseries\NUMBER.}{ \bfseries\NAME}},
        spaceabove=\paraskip, 
        headpunct={. },
        headfont=\bfseries,
        bodyfont=\normalfont
        ]{mystyle-plain}

\makeatletter
\renewenvironment{proof}[1][\textit{Proof}]{\par
  \pushQED{\qed}%
  \normalfont \topsep.75\paraskip\relax
  \trivlist
  \item[\hskip\labelsep
        \itshape
    #1\@addpunct{.}]\ignorespaces
}{%
  \popQED\endtrivlist\@endpefalse
}
\makeatother

\usepackage{tikz}
\usepackage{mathdots}
\newcommand\NN{\mathbb N}
\newcommand\CC{\mathbb C}

\newcommand\ZZ{\mathbb Z}

\newcommand\ot{\otimes}
\renewcommand\to{\longrightarrow}
\renewcommand\phi{\varphi}

\newcommand\gl{\mathfrak{gl}}
\newcommand\gen{\mathsf{gen}}
\newcommand\std{\mathsf{std}}

\DeclareMathOperator\End{End}
\DeclareMathOperator\Specm{Specm}

\title{A new way to construct $1$-singular Gelfand-Tsetlin modules}

\author{Pablo Zadunaisky
\footnote{Instituto de Matem\'atica e Estat\'istica, Universidade de S\~ao
Paulo,  S\~ao Paulo SP, Brasil. \texttt{email:} pzadun@ime.usp.br.
The author is a FAPESP PostDoc Fellow, grant: 2016-25984-1 
S\~ao Paulo Research Foundation (FAPESP).}
}

\begin{document}
\maketitle

\begin{abstract}
We present a simplified way to construct the Gelfand-Tsetlin modules over 
$\gl(n,\CC)$ related to a $1$-singular GT-tableau defined in 
\cite{FGR-singular-gt}. We begin by reframing the classical construction of 
generic Gelfand-Tsetlin modules found in \cite{DFO-GT-modules}, showing 
that they form a flat family over generic points of $\CC^{\binom{n}{2}}$. We
then show that this family can be extended to a flat family over a variety
including generic points and $1$-singular points for a fixed singular pair
of entries. The $1$-singular modules are precisely the fibers over these
points.
\end{abstract}


\section{Introduction}
A constant theme throughout the work of Sergei Ovsienko was the notion of  
Gelfand-Tsetlin subalgebras and modules introduced for $\mathfrak{sl}(3,\CC)$ 
in \cite{DFO-gt-modules-sl3}, then more generally for $\gl(n,\CC)$ in 
\cite{DFO-GT-modules-original}, finally arriving at the more general 
definition of Harish-Chandra algebras and modules in 
\cite{DFO-GT-modules}. He returned several times to the subject and its 
applications to representation theory, as can be seen in the articles 
\cites{MO-verma-gt, Ovs-finiteness-gt, Ovs-strong-nilp, FMO-hc-for-yangians, 
FO-galois, FMO-gk-conjecture-gt, FOS-torsion, FOS-gt-categories, 
FO-fibers-gt-categories}. There were also many independent 
developments such as the construction of quantized versions of Gelfand-Tsetlin
modules \cite{MT-q-gt}, the study of Gelfand-Tsetlin algebras and modules over
orthogonal Lie algebras \cites{Maz-ort-gt, Maz-gt-ort-modules}, and the 
construction of several new families of infinite dimensional Gelfand-Tsetlin
modules, including applications to the classification of irreducible modules
over $\gl(n,\CC)$ \cites{FGR-irred-gt-sl3, FGR-irred-gt-gl, FGR-singular-gt},
just to name a few. It is a pleasure to contribute to the continuing 
development of this subject. 

The notion of a Gelfand-Tsetlin module (see Definition \ref{D:gt-modules}) has 
its origin in the classical article \cite{GT-modules}, where I. Gelfand and M. 
Tsetlin gave a presentation of all the finite dimensional irreducible 
representations of $\gl(n,\CC)$ in terms of certain tableaux, which have come 
to be known as \emph{Gelfand-Tsetlin tableaux}, or GT-tableaux for shot. The 
action of $\gl(n,\CC)$ over a tableau is given by rational functions in its
entries, and the poles of these rational functions form a nowhere dense set
in $\CC^{\binom{n}{2}}$; the formulas for this action are known as the 
Gelfand-Tsetlin formulas. Starting from this observation, Yu. Drozd, Ovsienko 
and V. Futorny introduced a large family of infinite dimensional 
$\gl(n,\CC)$-modules in \cite{DFO-GT-modules}. These modules have a basis 
parametrized by Gelfand-Tsetlin tableaux with complex coefficients such that 
no pattern is a pole for the rational functions appearing in the GT-formulae 
(such tableaux are called \emph{generic}, hence the name ``generic 
Gelfand-Tsetlin module''). 

In \cite{FGR-singular-gt}, Futorny, D. Grantcharov and L. E. Ram\'\i rez built 
new GT-modules starting from a $1$-singular tableaux (see Definition 
\ref{D:various-tableaux}). These $1$-singular modules are given by generators
and relations, and the action of $\gl(n,\CC)$ is given by new explicit 
formulas; the proof that these formulas indeed define an action of $\gl(n,\CC)$
requires long calculations. We present a new approach, which consist of first
building a ``universal generic GT-module'' (Definition \ref{T:generic-GT})
and then finding a ``universal $1$-singular module'' (Definition \ref{D:V-B})
as a submodule. Thus the $\gl(n,\CC)$-module structure is built in by 
construction on the singular modules, and the formulas for this action are
deduced from those of the generic case. 

Let us be a little more explicit. The set of poles $P$ of the original 
Gelfand-Tsetlin formulae can be described as the union of all integral 
translates of a certain finite hyperplane arrangement centered at the origin 
in $\CC^{\binom{n}{2}}$. Let $X$ be the complement of $P$, and let $A$ be the 
algebra of regular functions defined over $X$, so for every point $x \in X$
there is a generic Gelfand-Tsetlin module $V_x$. We show that these modules
form a flat family over $X$ by constructing a free $A$-module $V_A$ over which 
$\gl(n, \CC)$ acts by $A$-linear operators, such that $V_A \ot_A \CC_x \cong 
V_x$ for each $x \in X$ as $\gl(n,\CC)$-modules, see Theorem 
\ref{T:generic-GT}. Now fix two entries $(k,i)$ and $(k,j)$, let $H \subset P$ 
be the set of tableaux with $v_{k,i} = v_{k,j}$ and no other critical entries,
and let $B$ be the algebra of regular functions over $Y = X \cup H$. In 
Theorem \ref{T:V-B-U-stable} we show that the flat family over $X$ extends 
to a flat family over $Y$ by finding a full $B$-lattice $V_B \subset V_A$
which is stable by the action of $\gl(n,\CC)$. From this we immediately get 
for each $x \in X$ that $V_x \cong V_B \ot_B \CC_x$ as $\gl(n,\CC)$-modules, 
while for $y \in H$ we obtain a $\gl(n,\CC)$-module $V_y = V_B \ot_B \CC_y$. 
We finally show that $V_y$ is isomorphic to the corresponding $1$-singular 
module constructed by Futorny, Grantcharov and Ram\'\i rez, and recover the
action of the Gelfand-Tsetlin subalgebra of $U(\gl(n,\CC))$ on it. 

\section{Gelfand-Tsetlin Tableaux}
We begin by quoting some results from \cite{FGR-singular-gt}*{section 3}. We 
direct the reader to that article for proofs or references.

Fix $n \in \NN_{\geq 2}$, and set $N = \frac{n(n+1)}{2}$. A 
\emph{Gelfand-Tsetlin tableau} of size $n$ (or GT-tableau of size $n$, for 
short) is an array with $N$ complex entries of the form

\begin{tikzpicture}
\node (n1) at (-2,2.5) {$v_{n,1}$};
\node (n2) at (-1,2.5) {$v_{n,2}$};
\node (ndots) at (0,2.5) {$\cdots$};
\node (nn-1) at (1,2.5) {$v_{n,n-1}$};
\node (nn) at (2,2.5) {$v_{n,n}$};

\node (n-11) at (-1.5,2) {$v_{n-1,1}$};
\node (n-1dots) at (0,2) {$\cdots$};
\node (n-1n-11) at (1.5,2) {$v_{n-1,n-1}$};

\node (dots1) at (-1,1.625) {$\ddots$};
\node (dots2) at (0,1.5) {$\cdots$};
\node (dots3) at (1,1.625) {$\iddots$};

\node (21) at (-.5,1) {$v_{2,1}$};
\node (22) at (.5,1) {$v_{2,2}$};
\node (11) at (0,.5) {$v_{1,1}$};

\node (A) at (-3.5, 2.75) {};
\node (B) at (3.5, 2.75) {};
\node (C) at (0,0) {};

\end{tikzpicture}

The set of all GT-tableaux can be identified with $\CC^N$ by indexing the 
entries of a point $v \in \CC^N$ by $\{(k,i) \mid 1 \leq i \leq k \leq n\}$. 
We fix one particular enumeration which will be useful later. Start by 
$\lambda_{n,n}$, which we denote by $x_1$; then, looking at all entries with 
second coordinate $n-1$, we enumerate them by taking $x_2 = 
\lambda_{n-1,n-1}$, then $x_3 = \lambda_{n,n-1}$; next we take the elements 
with second coordinate $n-2$ starting by $x_4 = \lambda_{n-2,n-2}$ and moving 
in the northwest direction. Explicitly, setting $\phi(i,j) = (i-j+1) + 
\frac{(n-j)(n-j+1)}{2}$ we write $x_{\phi(i,j)} = \lambda_{i,j}$. The 
following figure shows the enumeration corresponding to $n = 3$. 

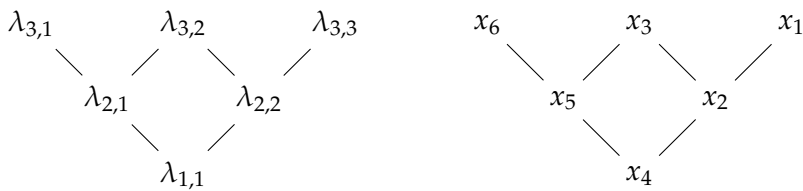
\begin{tikzpicture}
\node (31) at (0,3) {$\lambda_{3,1}$};
\node (32) at (2,3) {$\lambda_{3,2}$};
\node (33) at (4,3) {$\lambda_{3,3}$};
\node (21) at (1,2) {$\lambda_{2,1}$};
\node (22) at (3,2) {$\lambda_{2,2}$};
\node (11) at (2,1) {$\lambda_{1,1}$};

\node (31a) at (6,3) {$x_6$};
\node (32a) at (8,3) {$x_3$};
\node (33a) at (10,3) {$x_1$};
\node (21a) at (7,2) {$x_5$};
\node (22a) at (9,2) {$x_2$};
\node (11a) at (8,1) {$x_4$};

\draw (33) -- (22) -- (11)  (32) -- (21);
\draw (22) -- (32)  (11) -- (21) -- (31);

\draw (33a) -- (22a) -- (11a)  (32a) -- (21a);
\draw (22a) -- (32a)  (11a) -- (21a) -- (31a);
\end{tikzpicture}

We will denote by $T(v)$ the tableau corresponding to $v \in \CC^N$.
\begin{Definition}
\label{D:various-tableaux}
We say that a point $v \in \CC^N$, or the corresponding tableau $T(v)$, is:
\begin{itemize}
\item \emph{integral} if $v \in \ZZ^{N}$;

\item \emph{standard} if for all $1 \leq i < j \leq k \leq n$ the difference 
$v_{k,i}-v_{k-1,i} \in \ZZ_{\geq 0}$ and $v_{k-1,i}-v_{k,i+1} \in \ZZ_{>0}$;

\item \emph{generic} if for all $1 \leq i < j \leq k < n$ the difference 
$v_{k,i}-v_{k,j}$ is not in $\ZZ$;

\item \emph{singular} if it is a non-generic tableau, i.e. there is a pair of 
entries in the same row differing by an integer; if there is exactly one such 
pair then we say that the tableau is \emph{$1$-singular};

\item \emph{critical} if there are two entries in the same row which are 
equal; if there is exactly one such pair then we say the tableau is
\emph{$1$-critical}.
\end{itemize}
\end{Definition}
The enumeration of the $\lambda_{k,i}$ was chosen so all the equations defining
standard tableaux are of the form $x_t - x_s > 0$ or $x_t - x_s \geq 0$. 
Drawing a graph analogous to the one shown above with the $x_i$ as vertices, 
it is clear that we only get equations $x_t - x_s$ with $t>s$ and $x _t$ and 
$x_s$ joined by an edge.

We denote the set of standard points by $\CC^N_\std$, and the set of generic 
points by $\CC^N_\gen$. Any pair of two entries differing by an integer will
be called a \emph{singular pair}, and if they are equal they will be called
a \emph{critical pair}. 

\section{Gelfand-Tsetlin modules}
In this section, we recall the general notions of Gelfand-Tsetlin algebras and 
modules, and review the classical construction of generic Gelfand-Tsetlin 
modules due to Drozd, Futorny and Ovsienko. Our proof follows the outline 
given by V. Mazorchuk in \cite{Maz-GT-cats}*{Theorem 19}, but we believe it is 
worthwhile to include it since it allows us to introduce and become familiar 
with the main ingredients of our re-imagining of the construction of 
\cite{FGR-singular-gt}.

Recall we have fixed $n \in \NN$ and set $N = \frac{n(n+1)}{2}$. Let 
$\Lambda = \CC[\lambda_{i,j} \mid 1 \leq i \leq j \leq n]$ be a polynomial 
algebra in $N$ variables, and set $\Lambda_m = \CC[\lambda_{m,k} \mid 1 \leq k 
\leq m]$ for each $1 \leq m \leq n$. The group $S_m$ acts on $\Lambda_m$ 
permuting the variables in the obvious way. This induces an action of the 
group $G = S_n \times S_{n-1} \times \cdots \times S_1$ on $\Lambda$. Let
\[
	\gamma_{m,k} = \sum_{i = 1}^m (\lambda_{m,i}+m-1)^k 
	\prod_{i \neq j} \left( 1 - \frac{1}{\lambda_{m,i} - \lambda_{m,j}}\right).
\]
Although it is not obvious, the $\gamma_{m,k}$ are algebraically independent 
polynomials and $\Lambda_m^{S_m} = \CC[\gamma_{m,k} \mid 1 \leq k \leq m]$. 

For each $m \in \NN$ set $U_m = U(\gl(m, \CC))$. We denote by $Z_m \subset U_m$
the center of $U_m$. Also we write $U$ for $U_n$. We get a chain of inclusions 
$U_1 \subset U_2 \subset \cdots \subset U_n$ induced by the maps sending 
standard generators $E_{i,j} \in \gl(m,\CC)$ to the corresponding $E_{i,j} 
\in \gl(m+1, \CC)$. The algebra $Z_m$ is a polynomial algebra on the generators
\[
	c_{m,k} = \sum_{(i_1, \ldots, i_k) \in [m]^k} E_{i_1,i_2} E_{i_2,i_3} 
		\cdots E_{i_k, i_1} \qquad \qquad 1 \leq k \leq m,
\]
and there is an embedding $Z_m \to \Lambda_m$ given by $c_{m,k} \mapsto 
\gamma_{m,k}$. We write $\Gamma = \CC[c_{m,k} \mid 1 \leq k \leq m \leq n] 
\subset U$, which is the algebra generated by $\bigcup_{k=1}^n Z_k$. The 
$c_{m,k}$ are algebraically independent and hence $\Gamma$ is isomorphic to a 
polynomial algebra in $N$ generators. Thus $\Gamma$ is isomorphic to 
$\Lambda^G$.

\begin{Definition}[\cite{DFO-GT-modules}*{section 2.1}]
\label{D:gt-modules}
The algebra $\Gamma$ is called the \emph{Gelfand-Tsetlin subalgebra} of 
$U(\gl(n,\CC))$. A finitely generated $U$-module is called a 
\emph{Gelfand-Tsetlin module} if
\[
	M = \bigoplus_{\mathfrak m \in \Specm \Gamma} M (\mathfrak m),
\] 
where $M(\mathfrak m) = \{v \in M \mid \mathfrak m^k v = 0 \mbox{ for some } k 
\geq 0\}$.
\end{Definition}

\subsection{Irreducible representations of $\gl(n,\CC)$}
For each $1 \leq i \leq k \leq n$ set
\begin{align*}
p_{k,i}^\pm(\lambda) 
	&= \prod_{j = 1}^{k\pm1}(\lambda_{k,i} -\lambda_{k\pm1,j}); &
q_{k,i}(\lambda)
	&= \prod_{j \neq i} (\lambda_{k,i} - \lambda_{k,j}). \\
|\lambda|_k &= \lambda_{k,1} + \lambda_{k,2} + \cdots \lambda_{k,k}.
\end{align*}
The following is a classical result due to Gelfand and Tsetlin. 
\begin{Theorem}[\cite{GT-modules}]
\label{GT}
Let $\lambda = (\lambda_1, \ldots, \lambda_n)$ be a dominant integral weight 
of $\gl(n,\CC)$, and let 
\begin{align*}
V(\lambda) 
	&= \langle T(v) \mid v \in \CC^N_{\std} 
		\mbox{ and } v_{n,1} = \lambda_1, v_{n,2} = \lambda_2 - 1, \ldots, 
		v_{n,n} = \lambda_{n} - n +1
	\rangle_\CC
\end{align*}
(by convention, if $v$ is non-standard then $T(v) = 0$ in $V(\lambda)$).
The vector space $V(\lambda)$ can be endowed with a $\gl(n,\CC)$-module 
structure, with the action of the canonical generators given by
\begin{align*}
E_{k,k+1} T(v) 
	&= - \sum_{i=1}^k \frac{p^+_{k,i}(v)}{q_{k,i}(v)} T(v + \delta^{k,i}), \\
E_{k+1,k} T(v) 
	&= \sum_{i=1}^k \frac{p^-_{k,i}(v)}{q_{k,i}(v)} T(v - \delta^{k,i}), \\
E_{k,k} T(v)
	&= (|v|_k - |v|_{k-1} + k -1) T(v),
\end{align*}
where $\delta^{k,i}$ is the element of $\ZZ^N$ with a $1$ in position $(k,i)$
and $0$'s elsewhere. Furthermore, for each $1 \leq k \leq m \leq n$ we have 
$c_{m,k} T(v) = \gamma_{m,k}(v) T(v)$. 
\end{Theorem}
It is immediate to check that with this structure, $V(\lambda)$ is an 
irreducible finite dimensional representation of maximal weight $\lambda$, so 
this theorem provides an explicit presentation of all finite dimensional 
simple $\gl(n,\CC)$-modules. The formulas for the action of the generators of 
$\gl(n,\CC)$ in the previous theorem are known as the \emph{Gelfand-Tsetlin 
formulas}.

\subsection{Generic Gelfand-Tsetlin Modules}
Let $\ZZ_0^N \subset \ZZ^N$ be the set of vectors of $\CC^N$ with integral 
entries such that $v_{n,i} = 0$ for all $1 \leq i \leq n$. For $v \in \CC^N$
set 
\[
	V(T(v)) = \langle T(v+z) \mid z \in \ZZ^N_0 \rangle_\CC.
\]
We now quote \cite{DFO-GT-modules}*{section 2.3}.
\begin{Theorem}
\label{T:generic-GT}
Suppose $v \in \CC_\gen^N$. Then the vector space $V(T(v))$ can be endowed 
with the structure of a $\gl(n,\CC)$-module with the action of the canonical 
generators given by the Gelfand-Tsetlin formulas. Furthermore, for each $1 
\leq k \leq m \leq n$ we have $c_{m,k} T(w) = \gamma_{m,k}(w) T(w)$, so 
$V(T(v))$ is a Gelfand-Tsetlin module.
\end{Theorem}

We will now reprove Theorem \ref{T:generic-GT}. We begin with a technical 
lemma. Recall that we have identified $\CC^N$ with $\Specm \Lambda$, so
we identify rational functions over $\CC^N$ with elements of the field
$\CC(\lambda_{k,i} \mid 1 \leq i \leq k \leq n)$.
\begin{Lemma}
\label{L:rational-zeroes}
Let $W \subset \ZZ^N_0$ be a nonempty finite set and let $S_W = \bigcap_{w \in 
W} \CC^N_\std - w$. Let $F \in \CC(\lambda_{k,i} \mid 1 \leq i \leq k \leq n)$ 
be a rational function without poles in $S_W$. If $F(v) = 0$ for all $v \in 
S_W$, then $F = 0$.
\end{Lemma}
\begin{proof}
Explicitly, $v$ lies in $S_W$ if and only if
\begin{align*}
v_{k,i} - v_{k-1,i} + \min_{w \in W}\{w_{k,i} - w_{k-1,i}, 0\} + 1 
	&\in \ZZ_{> 0}, \\
v_{k-1,i} - v_{k,i+1} + \min_{w \in W}\{w_{k-1,i} - w_{k,i+1}\}
	&\in \ZZ_{>0},
\end{align*}
for all $1 \leq i \leq n$. With the enumeration fixed in the previous section,
and writing $x_{\phi(i,j)}$ for $\lambda_{i,j}$, these inequalities are all of 
the form $x_t - x_s - r_{t,s} \in \ZZ_{>0}$ with $t > s$ and $r_{t,s} \in \ZZ$ 
(notice that not all pairs $(t,s)$ appear). Thus fixing $x_1(v) = 0$, we can 
build recursively an element $v \in S_W$, in particular $S_W$ is not empty.

For each $1 \leq s \leq n$ let $w_s \in \ZZ^N$ be such that $x_t(w_s) = 1$ if 
$t \geq s$, and $x_t(w_s) = 0$ if $t < s$. Now $x \in S_W$ implies $x + r w_s 
\in S_W$ for all $r \in \NN$, and applying this to the element $v \in S_W$ we 
found before, we get that $v + \sum_{s=1}^N v + r_s w_s \in S_W$ whenever $r_s 
\in \NN$, and each of these points is a zero of $F$.

For each $1 \leq s \leq n$ let $e_s \in \ZZ^N$ be such that $x_s(e_s) = 1$
and $x_t(e_s) = 0$, and let $C: \CC^N \to \CC^N$ be the affine transformation
defined by $C(0) = v$, $C(e_s) = w_s$. Then $F \circ C$ is a rational function
without poles in $\NN^N$ and furthermore $F \circ C(\NN^n) = 0$. This implies
that $F \circ C = 0$, which in turn implies $F = 0$.
\end{proof}

Let $A$ be the algebra of regular functions defined over $\CC^N_\gen$. The 
algebra $A$ contains all the rational functions appearing in the 
Gelfand-Tsetlin formulas. Consider the action of $\ZZ^N$ over $\CC^N$, given 
by $v^z = v+z$. This induces an action on $A$, given by $(z \cdot f)(v) = 
f(v^z)$. We will sometimes write $f(\lambda^z)$ for $(z \cdot f)$. 

\begin{Proposition}
\label{P:universal-generic-GT-module}
Let $V_A$ be the free $A$-module with basis $\{T(z) \mid z \in \ZZ^N_0\}$.
The $A$-module $V_A$ can be endowed with the structure of a $U$-module 
with the action of the canonical generators given by
\begin{align*}
E_{k,k+1} T(z) 
	&= - \sum_{i=1}^k \frac{p^+_{k,i}(\lambda^z)}{q_{k,i}(\lambda^z)} 
		T(z + \delta^{k,i}); \\
E_{k+1,k} T(z) 
	&= \sum_{i=1}^k \frac{p^-_{k,i}(\lambda^z)}{q_{k,i}(\lambda^z)} 
		T(z - \delta^{k,i}); \\
E_{k,k} T(z)
	&= (|\lambda^z|_k - |\lambda^z|_{k-1} + k -1) T(z).
\end{align*}
Furthermore, for each $1 \leq k \leq m \leq n$, we have $c_{m,k} T(z) = 
\gamma_{m,k}(\lambda^z) T(z)$.
\end{Proposition}
\begin{proof}
Let $F$ be the free $\CC$-algebra generated by $X_{k,k+1}, X_{k+1,k}$ for 
$1 \leq k < n$, and $X_{k,k}$ for $1 \leq k \leq n$; there is an obvious map
$\phi: F \to U$. Replacing $E$ by $X$ in the formulas in the statement, we 
certainly get an $F$-module structure on $V_A$, so there is an algebra map $F 
\to \End_A V_A$, and we must show that this map factors through $U$.

Let $a \in F$. Then for each $w \in \ZZ^N_0$ there exists a rational function 
$f_{a,w} \in A$ such that
\[
	aT(z) = \sum_{w \in \ZZ^N_0} f_{a,w}(\lambda^z) T(z+w)
\]
for all $z \in \ZZ^N_0$. The sum is of course over a finite subset of 
$\ZZ^N_0$, which we call $W$. By Lemma \ref{L:rational-zeroes} the rational 
function $f_{a,w}$ is determined by its values on $S_W = \bigcap_{w \in W} 
\CC_\std^N-w$.

Fix $v \in S_W$ and let $\lambda = (v_{n,1}, v_{n,1} + 1, \ldots, v_{n,n} + 
n -1)$. Notice that $\lambda$ is the $n$-th row of $v+x$ for all $x \in S_W$ 
since these are all elements of $\ZZ^N_0$, and this in turn implies that 
$\lambda$ is an integral dominant weight of $\gl(n,\CC)$ and we can
consider the representation $V(\lambda)$ as defined in Theorem \ref{GT}. By 
construction the set $\{T(v+w) \mid w \in W\}$ consists of nonzero linearly 
independent elements of $V(\lambda)$, and by construction
\[
	\phi(a) T(v) = \sum_{w \in W} f_{a,w}(v) T(v+w).
\]
Thus if $\phi(a) = 0$ then $f_{a,w}(v) = 0$ for all $v \in S_W$, so $f_a(-,w) 
= 0$. This implies that the formulas in the statement indeed define a 
$U$-module structure on $V_A$. Furthermore, if $\phi(a) = c_{m,k}$ for
some $k$ and $m$ then $f_{a,w}(v) = 0$ for all $w \neq 0$ and $f_{a,0}(v) = 
\gamma_{m,k}(v)$, so again by Lemma \ref{L:rational-zeroes} the action of the
$c_{m,k}$ is the one given in the statement.
\end{proof}

\begin{proof}[Proof of Theorem \ref{T:generic-GT}]
\label{GT-generic-proof}
If $v \in \CC_\gen^N$ then the map $f \in A \mapsto f(v) \in \CC$ is well
defined, and induces a one-dimensional representation of $A$ which we denote
$\CC_v$. Now $V_A \ot_A \CC_v$ is a $U$-module, with the action of $U$ induced
by its action on $V_A$, and furthermore it is isomorphic to $V(T(v))$ as 
$\CC$-vector space through the assignation $1 \ot_A T(z) \mapsto T(v+z)$. 
Thus $V(T(v))$ gets a $U$-module structure, which by construction is given by
the Gelfand-Tsetlin formulas. 
\end{proof}

\section{The construction of $1$-singular Gelfand-Tsetlin modules}
We will give a new construction of Futorny, Grantcharov and Ramírez's 
$1$-singular Gelfand-Tsetlin modules, which will follow the same spirit as the 
construction of generic Gelfand-Tsetlin modules presented in the previous 
section. Fix $k,i,j \in \NN$ with $1 \leq i < j \leq k <n$. From now a 
$1$-critical point, will be a critical point $v$ whose only critical pair is 
$v_{k,i}, v_{k,j}$. We set $x = \lambda_{k,i}, y = \lambda_{k,j}$. 

Let $B \subset A$ be the algebra of consisting of functions in $A$ without 
poles in the set of $1$-critical points; for example $1/(x-y)$ lies in $A$
but not in $B$. The idea behind our construction is finding a $B$-lattice $V_B 
\subset V_A$ which is stable by the action of $U$. Once we have found this 
lattice, the construction of the $U$-module associated to a $1$-singular point 
$v$ can go as before: take $\CC_v$ to be the $1$-dimensional $B$-module 
associated to $v$ and set $V(T(v)) = V_B \ot_B \CC_v$, which inherits its 
$U$-module structure from $V_B$.

Let $I = \{(r,s) \mid 1 \leq s \leq r \leq n\}$, and let $\tau: I \to I$ be 
the involution that interchanges $(k,i)$ and $(k,j)$, while leaving all other
elements of $I$ fixed. By abuse of notation, we also denote by $\tau$ the 
linear transformation $\tau: \CC^N \to \CC^N$ given by $\tau(\delta^{r,s}) =
\delta^{\tau(r,s)}$ for all $(r,s) \in I$, and also the algebra automorphism
$\tau: B \to B$ induced by the assignation $\tau(\lambda_{r,s}) = 
\lambda_{\tau(r,s)}$. In each case $\tau^2$ is the identity map.

\begin{Definition}
\label{D:V-B}
For each $z \in \ZZ^N_0$ we define 
\begin{align*}
S(z)
	&= \frac{T(z) + T(\tau(z))}{2},
&A(z)
	&= \frac{T(z) - T(\tau(z))}{2(x-y)}.
\end{align*}
We define $V_B \subset A$ to be the $B$-module generated by $\{S(z), A(z) \mid
z \in \ZZ^N_0\}$.
\end{Definition}

As stated above, we will show that $V_B$ is stable by the action of $U$.
Notice that $T(z) = S(z) + (x-y)A(z)$, so $T(z) \in V_B$ for any $z \in 
\ZZ^N_0$. Notice also that $S(\tau(z)) = S(z)$ and $A(\tau(z)) = - A(z)$; 
in particular if $z = \tau(z)$ then $S(z) = T(z)$ and $A(z) = 0$. It follows 
that $a T(z) + b T(\tau(z)) = (a+b) S(z) + (a-b)(x-y)A(z)$ for all $a, b \in 
A$.

Let $D: B \to B$ be the differential operator $\frac{1}{2} \left(
\frac{\partial}{\partial x} - \frac{\partial}{\partial y} \right)$. The 
following lemma can be verified by direct calculations. 
\begin{Lemma}
\label{L:helpful}
Let $f \in B$.
\begin{enumerate}
\item 
\label{derivative}
The rational function $\frac{f - \tau \cdot f}{x-y}$ lies in $B$, and is equal
to $2 D(f)$.

\item 
\label{action1}
For each $z \in \ZZ^N$ we have $\tau \cdot f(\lambda^z) = (\tau \cdot f)
(\lambda^{\tau(z)})$. 

\item 
\label{action2}
We have $\tau \cdot p_{t,s}^\pm = p_{\tau(t,s)}^\pm$ and $\tau \cdot 
q_{t,s} = q_{\tau(t,s)}$.
\end{enumerate}
\end{Lemma}

\begin{Theorem}
\label{T:V-B-U-stable}
The $B$-module $V_B$ is a $U$-submodule of $V_A$.
\end{Theorem}
\begin{proof}
We will check by direct inspection that $V_B$ is stable by the action of the
canonical generators of $\gl(n,\CC)$. First, since $|\lambda^z|_t = 
|\lambda^{\tau(z)}|_t$ for all $1 \leq t \leq n$, both $T(z)$ and $T(\tau(z))$
are eigenvectors of $E_{t,t}$ of eigenvalue $|\lambda^z|_t - |\lambda^z|_{t-1} 
+ t-1$, so
\begin{align*}
E_{t,t} S(z) 
	&= (|\lambda^z|_t - |\lambda^z|_{t-1} + t-1) S(z), \\
E_{t,t} A(z) 
	&= (|\lambda^z|_t - |\lambda^z|_{t-1} + t-1) A(z),
\end{align*}
and $V_B$ is not only stable by the action of the $E_{t,t}$ with $1 \leq t 
\leq n$, but also a weight module.

Let us study the action of $E_{t,t+1}$ on $S(z)$. By definition we get
\begin{align*}
E_{t,t+1} S(z)
	&= - \frac{1}{2}\sum_{s=1}^t 
			\frac{p_{t,s}^+(\lambda^z)}{q_{t,s}(\lambda^z)}T(z + \delta^{t,s}) 
			+ \frac{p_{t,s}^+(\lambda^{\tau(z)})}{q_{t,s}(\lambda^{\tau(z)})}
				T(\tau(z) + \delta^{t,s}).
\end{align*}
Now since $T(\tau(z) + \delta^{t,s}) = T(\tau(z+\delta^{\tau(t,s)}))$, we can 
rewrite this last expression as 
\begin{align*}
- \frac{1}{2}\sum_{s=1}^t 
	\frac{p_{t,s}^+(\lambda^z)}{q_{t,s}(\lambda^z)}&T(z + \delta^{t,s}) 
	+ \frac{p_{\tau(t,s)}^+(\lambda^{\tau(z)})}
		{q_{\tau(t,s)}(\lambda^{\tau(z)})} T(\tau(z + \delta^{t,s})) \\
= -\frac{1}{2} \sum_{s=1}^t
	&\left( 
		\frac{p_{t,s}^+(\lambda^z)}{q_{t,s}(\lambda^z)} + 
		\frac{p_{\tau(t,s)}^+(\lambda^{\tau(z)})}
			{q_{\tau(t,s)}(\lambda^{\tau(z)})}
	\right) S(z+\delta^{t,s}) \\
&+
	\left( 
		\frac{p_{t,s}^+(\lambda^z)}{q_{t,s}(\lambda^z)} - 
		\frac{p_{\tau(t,s)}^+(\lambda^{\tau(z)})}
			{q_{\tau(t,s)}(\lambda^{\tau(z)})}
	\right) (x-y)A(z+\delta^{t,s}).	
\end{align*}
By definition $1/q_{t,s}(\lambda^z) \in B$ unless $(t,s) \in \{(k,i), (k,j)\}$
and $z = \tau (z)$. Hence the coefficients in the equation above lie in $B$, 
except perhaps when $z = \tau(z)$ and $t = k$, in which case it is not clear 
that the coefficients of $S(z + \delta^{k,i}) = S(z+\delta^{k,j})$ and $A(z + 
\delta^{k,i}) = -A(z + \delta^{k,j})$ lie in $B$. Set $q_{k,i}^*(\lambda^z) = 
q_{k,i}(\lambda^z)/(x-y)$, so the inverse of $q_{k,i}^*(\lambda^z)$ lies in 
$B$. The coefficient of $A(z + \delta^{k,i})$ can be written as 
\begin{align*}
2\left( 
		\frac{p_{k,i}^+(\lambda^z)}{q^*_{k,i}(\lambda^z)} + 
		\frac{p_{k,j}^+(\lambda^z)}{q_{k,j}^*(\lambda^z)}
	\right) \in B,
\end{align*}
where the factor $2$ arises from the fact that $A(z + \delta^{k,i}) = -A(z + 
\delta^{k,j})$. Now considering the coefficient of $S(z+\delta^{k,i})$, we may
use Lemma \ref{L:helpful} to rewrite it as 
\[
2\left(
\frac{p^+_{k,i}(\lambda^z)}{q_{k,i}(\lambda^z)} 
+ \frac{p^+_{k,j}(\lambda^{\tau(z)})}{q_{k,j}(\lambda^{\tau(z)})}
\right)
= \frac{2}{(x-y)} \left(
	\frac{p^+_{k,i}(\lambda^z)}{q_{k,i}^*(\lambda^z)} -
		\tau \cdot \frac{p^+_{k,i}(\lambda^z)}{q^*_{k,i}(\lambda^z)}
	\right)
= 4D\left(
	\frac{p^+_{k,i}(\lambda^z)}{q_{k,i}^*(\lambda^z)}
	\right).
\]
Thus $E_{t,t+1} S(z) \in V_B$ for all $t$ and all $z$. A similar analysis 
shows that $E_{t+1,t} S(z) \in V_B$.

Following the same reasoning we find that 
\begin{align*}
E_{t,t+1} A(z)
	&= - \frac{1}{2} \sum_{s=1}^t 
		\left(
			\frac{p_{t,s}^+(\lambda^z)}{q_{t,s}(\lambda^z)} - 
			\frac{p_{\tau(t,s)}^+(\lambda^{\tau(z)})}
				{q_{\tau(t,s)}(\lambda^{\tau(z)})}
		\right) \frac{1}{x-y} S(z + \delta^{t,s}) \\
	&\qquad \qquad + \left( 
		\frac{p_{t,s}^+(\lambda^z)}{q_{t,s}(\lambda^z)} + 
		\frac{p_{\tau(t,s)}^+(\lambda^{\tau(z)})}
			{q_{\tau(t,s)}(\lambda^{\tau(z)})}
	\right) A(z+\delta^{t,s}).
\end{align*}
If $z = \tau(z)$ then $A(z) = 0$, so we only have to consider the case
$z \neq \tau(z)$, in which case the formula above equals
\begin{align*}
- \frac{1}{2} \sum_{s=1}^t 
	2D\left(
		\frac{p_{t,s}^+(\lambda^z)}{q_{t,s}(\lambda^z)}
	\right) S(z + \delta^{t,s}) + \left( 
		\frac{p_{t,s}^+(\lambda^z)}{q_{t,s}(\lambda^z)} + 
		\frac{p_{\tau(t,s)}^+(\lambda^{\tau(z)})}
			{q_{\tau(t,s)}(\lambda^{\tau(z)})}
	\right) A(z+\delta^{t,s}). 
\end{align*}
Since all coefficients lie in $B$, we see that $E_{t,t+1} A(z) \in V_B$.
A similar formula can be obtained for $E_{t+1,t} A(z) \in V_B$, which completes
the proof that $V_B$ is stable by the action of $U$.
\end{proof}

The formulas obtained in the proof of Theorem \ref{T:V-B-U-stable} can be misleading. 
For example, notice that for $t \neq k-1,k$ the coefficient of each $A(z + 
\delta^{t,s})$ in $E_{t,t+1} S(z)$ is zero. A case by case analysis reveals 
much simpler formulas, but we postponed for the sake of brevity and because 
we will add one further simplification. Since we are going to study the 
specialization of $V_B$ at $1$-critical points, we might as well consider the 
coefficients not in $B$ but in $B' = B/(x-y)$. Reduction modulo $(x-y)$ of
course simplifies the formulas. 

\begin{Definition}
Set $\overline V = B' \ot_B V_B$. We refer to $\overline V$ as the 
\emph{universal $1$-singular Gelfand-Tsetlin module}.
Given a $1$-singular point $v \in \CC^N$ with $v_{k,i} = v_{k,j}$, we denote by
$\CC_v$ the one-dimensional representation of $B'$ induced by $v$, and by 
$V(T(v))$ the $U$-module $\overline V \ot_{B'} \CC_v = V \ot_B \CC_v$.
\end{Definition}
We write $\overline S(z)$ and $\overline A(z)$ for the images of $S(z)$ and 
$A(z)$ in $\overline V$, respectively. Also, we denote by $\overline 
p^\pm_{t,s}, \overline q_{t,s}$ the images of $p^{\pm}_{(t,s)}, 
\overline q_{t,s}$ in $B'$, respectively. First, notice that if $t \neq k$ or 
$z \neq \tau (z)$ the definitions imply that
\begin{align*}
\frac{p_{t,s}^\pm(\lambda^z)}{q_{t,s}(\lambda^z)} \equiv
		\frac{p_{\tau(t,s)}^\pm(\lambda^{\tau(z)})}
			{q_{\tau(t,s)}(\lambda^{\tau(z)})} \mod (x-y).
\end{align*}
Applying reduction modulo $(x-y)$ to the formulas obtained in the proof of
Theorem \ref{T:V-B-U-stable}, we obtain that if $t \neq k$ or $z \neq 
\tau(z)$ then
\begin{align*}
E_{t,t+1} \overline S(z)
	&= -\sum_{s=1}^t \frac{\overline p^+_{t,s}(\lambda^z)}
		{\overline q_{t,s}(\lambda^z)} 
		\overline S(z+\delta^{t,s}), 
&E_{t+1,t} \overline S(z)
	&= \sum_{s=1}^t \frac{\overline p^-_{t,s}(\lambda^z)}
		{\overline q_{t,s}(\lambda^z)} 
		\overline S(z-\delta^{t,s}).
\end{align*}
Writing as before $q_{k,r}^*(\lambda^z) = q_{k,r}(\lambda^z)/(x-y)$, and 
setting $\overline D(f) = \overline{D(f)}$ we also obtain formulas
\begin{align*}
E_{k,k+1} \overline S(z)
	&= -\sum_{s \neq i,j} \frac{\overline p^+_{k,s}(\lambda^z)}
		{\overline q_{k,s}(\lambda^z)} 
		\overline S(z+\delta^{k,s}) 
	- 2 \overline D\left(
		\frac{p^+_{k,i}(\lambda^z)}{q^*_{k,i}(\lambda^z)}
	\right) \overline S(z + \delta^{k,i}) \\
	& \qquad - 2 \left(
		\frac{\overline p^+_{k,i}(\lambda^z)}{\overline q^*_{k,i}(\lambda^z)}
	\right) \overline A(z+\delta^{k,i}); \\
E_{k+1,k} \overline S(z)
	&= \sum_{s \neq i,j} \frac{\overline p^-_{k,s}(\lambda^z)}
		{\overline q_{k,s}(\lambda^z)} 
		\overline S(z-\delta^{k,s}) 
	+ 2\overline D\left(
		\frac{p^-_{k,i}(\lambda^z)}{q^*_{k,i}(\lambda^z)}
	\right) \overline S(z - \delta^{k,i}) \\
	& \qquad + 2 \left(
		\frac{\overline p^-_{k,i}(\lambda^z)}{\overline q^*_{k,i}(\lambda^z)}
	\right) \overline A(z - \delta^{k,i}).
\end{align*}
For the sake of comparison with the original construction, we point out that
this formulas can be summarized as
\begin{align*}
E_{t,t+1} S(z) 
	&= -\sum_{s=1}^t \overline D\left(
		\frac{p^+_{t,s}(\lambda^z)}{q^*_{t,s}(\lambda^z)}
	\right) \overline S(z + \delta^{t,s})
	+ \left(
		\frac{\overline p^+_{t,s}(\lambda^z)}{\overline q^*_{t,s}(\lambda^z)}
	\right) \overline A(z + \delta^{t,s}) \\
E_{t+1,t} S(z) 
	&= \sum_{s=1}^t \overline D\left(
		\frac{p^-_{t,s}(\lambda^z)}{q^*_{t,s}(\lambda^z)}
	\right) \overline S(z + \delta^{t,s})
	+ \left(
		\frac{\overline p^-_{t,s}(\lambda^z)}{\overline q^*_{t,s}(\lambda^z)}
	\right) \overline A(z + \delta^{t,s})
\end{align*}

In the case of $\overline A(z)$, reduction modulo $(x-y)$ of the formulas 
already found gives
\begin{align*}
E_{t,t+1} \overline A(z)
	&= - \sum_{s=1}^t \overline D \left(
		\frac{p_{t,s}^+(\lambda^z)}{q_{t,s}(\lambda^z)}
	\right) \overline S(z + \delta^{t,s}) 
	+ \left( 
		\frac{\overline p_{t,s}^+(\lambda^z)}{\overline q_{t,s}(\lambda^z)}
	\right) \overline A(z+\delta^{t,s}), \\
E_{t+1,t} \overline A(z)
	&= \sum_{s=1}^t \overline D \left(
		\frac{p_{t,s}^-(\lambda^z)}{q_{t,s}(\lambda^z)}
	\right) \overline S(z - \delta^{t,s}) 
	+ \left( 
		\frac{\overline p_{t,s}^-(\lambda^z)}{\overline q_{t,s}(\lambda^z)}
	\right) \overline A(z-\delta^{t,s})
\end{align*}
With this formulas, the following theorem is clear.
\begin{Theorem}
Let $v \in \CC^N$ be a $1$-singular point with $v_{k,i} = v_{k,j}$. Then the
$U$-module $V(T(v)) = V_{B'} \ot_{B'} \CC_v$ is isomorphic to the $1$-singular
Gelfand-Tsetlin module defined in \cite{FGR-singular-gt}*{Theorem 4.11} 
through the map $1 \ot_{B'} \overline S(z) \mapsto T(v + z), 1 \ot_{B'} 
\overline A(z) \mapsto DT(v+z)$.
\end{Theorem}

We finish by studying the action of the Gelfand-Tsetlin algebra on $\overline 
V$. As before, we first consider its action on $V_B$, and in that case we get
\begin{align*}
c_{m,t} S(z)
	&= \gamma_{m,t}(\lambda^z)S(z); 
&c_{m,t} A(z)
	&= \gamma_{m,t}(\lambda^z)A(z) 
	&(m \neq k),
\end{align*}
and
\begin{align*}
c_{k,t} S(z)
	&= \frac{\gamma_{k,t}(\lambda^z) + \gamma_{k,t}(\lambda^{\tau(z)})}{2} S(z)
		+ (x-y)\frac{\gamma_{k,t}(\lambda^z) - \gamma_{k,t}(\lambda^{\tau(z)})}
		{2}A(z); \\
c_{k,t} A(z)
	&= \frac{\gamma_{k,t}(\lambda^z) + \gamma_{k,t}(\lambda^{\tau(z)})}{2}A(z)
		+ \frac{\gamma_{k,t}(\lambda^z) - \gamma_{k,t}(\lambda^{\tau(z)})}
		{2(x-y)}S(z).
\end{align*}
Now $\gamma_{t,s}$ is a symmetric polynomial, which implies that $\tau \cdot 
\gamma_{t,s} = \gamma_{t,s}$. Thus reducing modulo $x-y$ and using item 
\ref{derivative} of Lemma \ref{L:helpful} we obtain
\begin{align*}
c_{m,t} \overline S(z)
	&= \gamma_{m,t}(\lambda^z) \overline S(z);
&c_{m,t} \overline A(z)
	&= \gamma_{m,t}(\lambda^z) \overline A(z); 
	&(m \neq k),
\end{align*}
and
\begin{align*}
c_{k,t} \overline S(z)
	&= \overline \gamma_{k,t}(\lambda^z) \overline S(z); 
&c_{k,t} \overline A(z)
	&= \gamma_{k,t}(\lambda^z) \overline A(z)
		+ \overline D(\gamma_{k,t}(\lambda^z)) \overline S(z).
\end{align*}
From this we recover \cite{FGR-singular-gt}*{Corollary 4.13}.

\end{document}